\documentclass[a4paper, 12pt]{article} 
\usepackage{amssymb,latexsym,color} 

\newtheorem{dref}{Definition}[section] \newtheorem{lemma}[dref]{Lemma}
\newtheorem{theo}[dref]{Theorem} \newtheorem{prop}[dref]{Proposition}
\newtheorem{remark}[dref]{Remark} 
\newtheorem{cor}[dref]{Corollary}

\newenvironment{proof}{\par\noindent{{\bf Proof.}}}{\hfill$\Box$
\medskip}

\newcommand{\ekv}[2]{\begin{equation}\label{#1}#2\end{equation}}
\newcommand{\eekv}[3]{\begin{eqnarray}\label{#1}#2 \\ #3
\nonumber\end{eqnarray}}
\newcommand{\eeekv}[4]{\begin{eqnarray}\label{#1}#2 \\ #3
\nonumber\\#4\nonumber\end{eqnarray}}
\newcommand{\eeeekv}[5]{\begin{eqnarray}\label{#1}#2 \\ #3
\nonumber\\#4\nonumber\\#5\nonumber\end{eqnarray}}

\newcommand{\no}[1]{(\ref{#1})} 

\title{Counting zeros of holomorphic functions of exponential growth} 
\author{Johannes
Sj{\"o}strand \footnote{Ce travail a b\'en\'efici\'e d'une aide de l'Agence Nationale de la Recherche
portant la r\'ef\'erence ANR-08-BLAN-0228-01}\\\small Institut de Math\'ematiques de Bourgogne,
Universit\'e de Bourgogne\\ 
\small 9 avenue Alain Savary - BP 47870\\ 
\small 21078 Dijon cedex\\ \footnotesize
johannes.sjostrand@u-bourgogne.fr\\
\footnotesize and
UMR 5584 du CNRS}  
\date{}
\begin{document}
\maketitle
\abstract{We consider the number of zeros of holomorphic functions in a bounded domain that depend on a small parameter and satisfy an exponential upper bound near the boundary of the domain and similar lower bounds at finitely many points along the boundary. Roughly the number of such zeros is $(2\pi h)^{-1}$ times the integral over the domain of the laplacian of the exponent of the dominating exponential. Such results have already been obtained by M.~Hager and by Hager and the author and they are of importance when studying the asymptotic distribution of eigenvalues of elliptic operators with small random perturbations. In this paper we generalize these results and arrive at geometrically natural statements and natural remainder estimates.

\medskip\par\centerline{{\bf R\'esum\'e}} 

\smallskip\par Nous \'etudions le nombre de z\'eros dans un domaine born\'e d'une fonction holomorphe d\'ependant d'un petit param\`etre et v\'erifiant une borne exponentielle sup\'erieure pr\`es du bord et des bornes exponentielles inf\'erieures dans un nombre fini de points pr\`es du bord. Approximativement le nombre de ces z\'eros est \'egale \`a $(2\pi h)^{-1}$ fois l'int\'egrale sur le domaine du laplacien appliqu\'e \`a l'exposant de la borne exponentielle sup\'erieure. 
De tels r\'esultats ont d\'ej\`a \'et\'e obtenus par M.~Hager et par Hager et l'auteur dans des \'etudes sur la distribution asymptotique des valeurs propres d'op\'erateurs elliptiques avec des petites perturbations al\'eatoires. Dans ce travail on g\'en\'eralise ces r\'esultats pour arriver \`a des \'enonc\'es g\'eom\'etriquement naturels avec des estimations du reste naturelles.
}


\tableofcontents

\section{Introduction}\label{int}
\setcounter{equation}{0}

Starting with the works of M.~Hager \cite{Ha3,Ha1,Ha2} there has been a number of results (\cite{HaSj,Sj08a,Sj08b,DaHa09,Bo,BoSj}) that establish Weyl asymptotics of the eigenvalues of non-self-adjoint (pseudo)differential operators with small random perturbations, in the semi-classical limit and in the limit of large eigenvalues. A common feature here (as well as in many other non-self-adjoint spectral problems) is that one identifies the eigenvalues with the zeros of a holomorphic function $u(z)=u(z;h)$, $0<h\ll 1$ in a set $\Gamma \Subset {\bf C}$. The available information is an upper bound $|u(z;h)|\le \exp (\phi (z)/h)$ for $z$ near the boundary $\partial \Gamma $ as well as lower bounds 
$
|u(z_j;h)|\ge \exp (\phi (z_j)-\epsilon _j)/h
$, for finitely many points $z_j=z_j(h)$, $1\le j\le N(h)$ that are suitably distributed near the boundary. 

Hager \cite{Ha3,Ha1,Ha2} obtained a result of this type with the conclusion that the number of zeros in $\Gamma $ is $(2\pi h)^{-1}(\int_\Lambda \Delta \phi (z)L(dz)+o(1))$ in the limit of small $h$, when $\Gamma $ is independent of $h$
 with smooth boundary and $\phi $ is a $C^2$ function, also independent of $h$. In \cite{HaSj} we generalized this result by weakening the regularity assumptions on $\phi $. However, due to some logarithmic losses, we were not quite able to recover Hager's original result, and we still had a fixed domain $\Gamma $ with smooth boundary. 

In many spectral problems of the above type the domain should be allowed to depend on $h$, for instance, it could be a long thin rectangle, and the boundary regularity should be relaxed. 

In the present paper we have revisited systematically the proof of the counting proposition in \cite{HaSj} and obtained a general and quite natural result allowing an $h$ -dependent exponent $\phi $ to be merely continuous and the $h$-dependent domain $\Gamma $ to have Lipschitz boundary. The result generalizes the two earlier ones. By allowing suitable small changes of the points $z_j$ we also get rid of the logarithmic losses. The new results below allow some improvement in the spectral results of \cite{Sj08b} (see for instance \cite{Sj09}) and they seem to allow a better understanding of such results in general. We hope that future works will supply applications.

We next formulate the results. 
Let $\Gamma \Subset {\bf C}$ be an open set and let $\gamma =\partial
 \Gamma $ be the boundary of $\Gamma $. Let $r:\gamma \to ]0,\infty [$
 be a Lipschitz function of Lipschitz modulus $\le 1/2$:
\ekv{int.1}
{
\vert r(x)-r(y)\vert \le \frac{1}{2}\vert x-y\vert,\ x,y\in \gamma .
}
We further assume that $\gamma $ is Lipschitz in the following precise
sense, where $r$ enters:

There exists a constant $C_0$ such that for every $x\in \gamma $ there exist new affine coordinates $\widetilde{y}=(\widetilde{y}_1,\widetilde{y}_2)$ of the form $\widetilde{y}=U(y-x)$, $y\in {\bf C}\simeq {\bf R}^2$ being the old coordinates, where $U=U_x$ is orthogonal, such that the intersection of $\Gamma $ and the rectangle $R_x:= \{ y\in {\bf C} ;\, |\widetilde{y}_1|<r(x),\, |\widetilde{y}_2|<C_0r(x)\} $ takes the form \ekv{int.2} {\{ y\in R_x;\, \widetilde{y}_2>f_x(\widetilde{y}_1),\ |\widetilde{y}_1|<r(x)\} , } where $f_x(\widetilde{y}_1)$ is Lipschitz on $[-r(x),r(x)]$, with Lipschitz modulus $\le C_0$.

\par Notice that our assumption (\ref{int.2}) remains valid if we
decrease $r$. It will be convenient to extend the function to all of
${\bf C}$, by putting 
\ekv{int.3}
{
r(x)=\inf_{y\in \gamma }(r(y)+\frac{1}{2}|x-y|).
}
The extended function is also Lipschitz with modulus $\le
\frac{1}{2}$: 
$$
|r(x)-r(y)|\le \frac{1}{2}|x-y|,\ x,y\in {\bf C}.
$$
Notice that 
\ekv{int.4}
{
r(x)\ge \frac{1}{2}\mathrm{dist\,}(x,\gamma ),
}
and that 
\ekv{int.5}
{
|y-x|\le r(x)\Rightarrow \frac{r(x)}{2}\le r(y)\le \frac{3r(x)}{2}.
}

\par For simplicity, we shall also assume that $\Gamma $ is simply connected. The complete version of our result is: 
\begin{theo}\label{int1}
Let $\Gamma \Subset {\bf C}$ be simply connnected and have Lipschitz boundary $\gamma $ with an
associated Lipschitz weight $r$ as in (\ref{int.1}), (\ref{int.2}), 
(\ref{int.3}).
 Put $\widetilde{\gamma }_{\alpha r}=\cup_{x\in \gamma
}D(x,\alpha r(x))$ for any constant $\alpha >0$. Let $z_j^0\in \gamma $, $j\in {\bf Z}/N{\bf Z} $ be distributed along 
the  boundary in the positively oriented sense such that 
$$r(z_j^0)/4\le |z_{j+1}^0-z_j^0|\le r(z_j^0)/2 .$$ (Here ``4'' can be
replaced by any fixed constant $>2$.) 
Then there exists a constant
$C_1>0$ depending only on the constant $C_0$ in the assumption around (\ref{int.2}) such that if $z_j\in D(z_j^0,r(z^0_j)/(2C_1))$ we have the following:

\par Let $0<h\le 1$ and let $\phi $ be a continuous subharmonic function on $\widetilde{\gamma }_{r}$ with a distribution extension to $\Gamma \cup \widetilde{\gamma }_{r}$ that will be denoted by the same symbol. Then there exists a constant $C_2>0$ such that if $u$ is a holomorphic function on $\Gamma \cup \widetilde{\gamma }_{r}$ satisfying \ekv{int.6} { h\ln |u|\le \phi (z)\hbox{ on }\widetilde{\gamma }_{r}, } \ekv{int.7} { h\ln |u(z_j)|\ge \phi (z_j)-\epsilon _j,\hbox{ for }j=1,2,...,N, } where $\epsilon _j\ge 0$, then the number of zeros of $u$ in $\Gamma $ satisfies \eekv{int.8} { &&|\# (u^{-1}(0)\cap \Gamma )-\frac{1}{2\pi h}\mu (\Gamma )|\le } {&& \frac{C_2}{h}\left( \mu (\widetilde{\gamma }_{r})+\sum_1^N \left(\epsilon _j+\int _{D(z_j,\frac{r(z_j)}{4C_1})}|\ln \frac{|w-z_j|}{r(z_j)}|\mu (dw)\right) \right).  } Here $\mu :=\Delta \phi \in {\cal D}'(\Gamma \cup \widetilde{\gamma }_{r})$ is a positive measure on $\widetilde{\gamma }_{r}$ so that $\mu (\Gamma )$ and $\mu (\widetilde{\gamma }_{r})$ are well-defined. Moreover, the constant $C_2$ only depends on $C_0$ in (\ref{int.2}).  \end{theo} By observing that the average of $|\ln \frac{|w-z_j|}{r(z_j)}|$ with respect to the Lebesgue measure $L(dz_j)$ over $D(z_j^0,\frac{r(z_j^0)}{2C_1})$ is ${\cal O}(1)$, we can get rid of the logarithmic terms in Theorem \ref{int1}, to the price of making a suitable choice of $z_j=\widetilde{z}_j$, and we get: \begin{theo}\label{int2} Let $\Gamma $, $\gamma =\partial \Gamma $, $r$, $z_j^0$, $C_0$, $C_1$, $\phi $ be as in Theorem \ref{int1}. Then $\exists\,\widetilde{z}_j\in D(z_j^0,\frac{r(z_j^0)}{2C_1})$ such that if $h$, $u$ are as in Theorem \ref{int1}, satisfying (\ref{int.6}), and \ekv{int.9} { h\ln |u(\widetilde{z}_j)|\ge \phi (\widetilde{z}_j)-\epsilon _j,\ j=1,2,...,N, } instead of (\ref{int.7}), then \ekv{int.10} { |\# (u^{-1}(0)\cap \Gamma )-\frac{1}{2\pi h}\mu (\Gamma )| \le \frac{C_2}{h}(\mu (\widetilde{\gamma }_{r})+\sum \epsilon _j ).  } \end{theo}

\par Of course, if we already know that 
\ekv{int.11}
{
\int_{D(z_j,\frac{r(z_j)}{4C_1})}|\ln \frac{|w-z_j|}{r(z_j)}| \mu (dw)=
{\cal O}(1)\mu (D(z_j,\frac{r(z_j)}{4C_1})),
}
then we can keep $\widetilde{z}_j=z_j$ in (\ref{int.8}) and get
(\ref{int.10}). This is the case, if we assume that $\mu $ is
equivalent to the Lebesgue measure $L(d\omega )$ in the following sense:
\eekv{int.12}
{
&&\frac{\mu (dw)}{\mu (D(z_j,\frac{r(z_j)}{4C_2}))} \asymp 
\frac{L (dw)}{L (D(z_j,\frac{r(z_j)}{4C_2}))}
\hbox{ on }D(z_j,\frac{r(z_j)}{4C_2}),}
{&&\hbox{uniformly for }j=1,2,...,N.
}
Then we get,
\begin{theo}\label{d3}
Make the assumptions of Theorem \ref{int1} as well as (\ref{int.11}) or the
stronger assumption (\ref{int.12}). Then from (\ref{int.6}), (\ref{int.7}),
we conclude (\ref{int.10}).
\end{theo}
In particular, we recover the counting proposition of M.~Hager
\cite{Ha1, Ha2}, where $\Gamma $, $\phi $ are independent of $h$, $\gamma $ of class $C^\infty $ and $\phi \in C^\infty (\mathrm{neigh\,}(\gamma ))$ (and the replacement of ``$\infty $" by ``2'' is straight forward). Then $\mu
\asymp L$ and if we choose $r\ll 1$ constant and assume (\ref{int.6}),
(\ref{int.7}), we get from (\ref{int.9}):
\ekv{int.13}
{
|\# (u ^{-1}(0)\cap \Gamma )-\frac{1}{2\pi h}\mu (\Gamma )|
\le
\frac{\widetilde{C}}{h}(r+\sum_1^N \epsilon _j ).
}
Hager had $\epsilon _j=\epsilon $ independent of $j$,
$r=\sqrt{\epsilon }$, $N\asymp \epsilon ^{-1/2}$, so the remainder in
(\ref{int.13}) is ${\cal O}(\frac{\sqrt{\epsilon }}{h})$. The counting proposition in \cite{HaSj} can also be recovered.

\par There has been a considerable activity in the study of the zero set of random holomorphic functions where the Edelman Kostlan formula has similarities with the above results (and the earlier ones by Hager and others metnioned above) and where many further results have been obtained. See M.~Sodin \cite{So00}, Sodin, B.~Tsirelson \cite{SoTs05}, S.~Zrebiec \cite{Zr07}, B.~Shiffman, S.~Zelditch, S.~Zrebiec \cite{ShZeZr08}. The two last papers deal with holomorphic functions of several variables and it would be interest to see if our results also have extensions to the case of several variables. Our results are similar in spirit to classical results on zeros of entire functions, see Levin \cite{Le80}.

The outline of the  paper is the following:

\par In Section \ref{cz} we consider thin neighborhoods of the boundary where the width is variable and determined by the function $r$. We verify that we can find such neighborhoods with smooth boundary and estimate the derivatives of the boundary defining function. Then we develop some exponentially weighted estimates for the Laplacian in such domains in the spirit of what can be done for the Schr\"odinger equation (\cite{HeSj84}) and a large number of works in thin domains, see for instance \cite{GrJe96, BoEx04}. From that we also deduce pointwise estimates on the corresponding Green kernel.

\par In Section \ref{di} we prove the main results by following the general strategy of the proof of the corresponding result in \cite{HaSj} and carry out the averaging argument that leads to the elimination of the logarithms.

\par In Section \ref{ap} we consider as a simple illustration the zeros of sums of exponentials of holomorphic functions. These results can also be obtained with more direct methods, cf \cite{Da03, BlMa06, HiSj3a}.

Finally in Section \ref{ent} we establish a connection with classical results on zeros of entire functions.

\smallskip\par\noindent 
{\bf Acknowledgement} We have benefitted from interesting discussions with Scott Zrebiec at the Mittag-Leffler institute during the special program on complex analysis in the Spring of 2008.

\section{Thin neighborhoods of the boundary and weighted estimates} \label{cz}
\setcounter{equation}{0}

Let $\Gamma $, $\gamma =\partial \Gamma $, $r$ be as in the introduction.

Using a locally finite covering with discs $D(x,r(x))$ and a
subordinated partition of unity, it is standard to find a smooth function
$\widetilde{r}(x)$ satisfying 
\ekv{a.6}
{
\frac{1}{C}r(x)\le \widetilde{r}(x)\le r(x),\ |\nabla
\widetilde{r}(x)|\le \frac{1}{2},\ \partial ^\alpha
\widetilde{r}(x)={\cal O}(\widetilde{r}^{1-|\alpha |}),
}
where $C>0$ is a universal constant.
\par From now on, we replace $r(x)$ by $\widetilde{r}(x)$ 
and the drop the
tilde. (\ref{int.1}), (\ref{int.2}), (\ref{int.5}) remain valid and
(\ref{int.4}) remains valid in the weakened form:
$$
r(x)\ge \frac{1}{C}\mathrm{dist\,}(x,\gamma ),
$$
where $C>0$ is new constant.

\par Consider the signed distance to $\gamma $:
\ekv{a.7}
{
g(x)=\cases{\mathrm{dist\,}(x,\gamma ),\ x\in \Gamma \cr
  -\mathrm{dist\,}(x,\gamma ),\ x\in {\bf C}\setminus \Gamma }
}
Possibly after replacing $r$ by a small constant multiple of $r$ we deduce
from (\ref{int.2}) that for every $x\in \gamma $ there exists a
normalized constant real vector field $\nu =\nu _x$ (namely $\partial _{\widetilde{y}_1}$, cf. (\ref{int.2})) such that 
\ekv{a.8}
{
\nu (g)\ge \frac{1}{C} \hbox{ in }R_x,\quad C=(1+C_0^2)^{-1/2}. 
}

\par In the set $\cup_{x\in \gamma }R_x$, we consider the
regularized function 
\ekv{a.9}
{
g_{\epsilon }(x)=\int \frac{1}{(\epsilon r(x))^2}\chi
(\frac{x-y}{\epsilon r(x)})g(y) L(dy),
}
where $0\le \chi \in C_0^\infty (D(0,1))$, $\int \chi
(x)L(dx)=1$. Here $\epsilon >0$ is small and we notice that
$r(x)\asymp r(y)$, $g(y)={\cal O}(r(y))$, when $\chi ((x-y)/(\epsilon
r(x))\ne 0$. It follows that $g_\epsilon (x)={\cal O}(r(x))$ and more
precisely, since $g$ is Lipschitz, that 
\ekv{a.10}
{
g_\epsilon (x)-g(x)={\cal O}(\epsilon r(x)).
}
Differentiating (\ref{a.9}), we get 
\eekv{a.11}
{\nabla _xg_\epsilon (x)&=&(\nabla _xg)_\epsilon +
2 \int \frac{-\nabla r(x)}{\epsilon ^2r(x)^2}\chi (\frac{x-y}{\epsilon r(x)})\frac{g(y)}{r(x)}L(dy)
}
{&&+\int \frac{1}{\epsilon ^2r(x)^2}\chi '(\frac{x-y}{\epsilon
    r(x)})\cdot \frac{x-y}{\epsilon r(x)}(-\nabla
  r(x))\frac{g(y)}{r(x)}L(dy),}
where $(\nabla _xg)_\epsilon $ is defined as in (\ref{a.9}) with $g$
replaced by $\nabla _xg$. It follows that 
\ekv{a.12}{\nabla _xg_\epsilon (x)-(\nabla g)_\epsilon (x)={\cal O}(1)
\sup_{y\in D(x,\epsilon r(x))}\frac{|g(y)|}{r(x)}.}
In particular, $\nabla _xg_\epsilon ={\cal O}(1)$ and with $\nu =\nu
_x$:
\ekv{a.13}
{
\nu (g_\epsilon )(y)\ge \frac{1}{2C},\hbox{ when }y\in R_x,\hbox{ and }
\sup_{|z-y|\le \epsilon r(y)}|g(z)|\ll r(y).
}
Differentiating (\ref{a.11}) further, we get 
\ekv{a.14}
{
\partial ^\alpha g_\epsilon (x)={\cal O}_{\alpha }((\epsilon
r(x))^{1-|\alpha |}),\ |\alpha |\ge 1.
}

\par Let $C>0$ be large enough but independent of $\epsilon $. Put 
\ekv{a.15}
{
\widehat{\gamma }_{\epsilon ,C\epsilon r}=\{ x\in \cup_{y\in \gamma
}R_y;\, |g_\epsilon (x)|<C\epsilon r(x)\}.
}

\par If $C>0$ is sufficiently large, then in the coordinates
associated to (\ref{int.2}), $\widehat{\gamma }_{\epsilon ,C\epsilon r}$ takes
the form
\ekv{a.16}
{
f_x^-(\widetilde{y}_1)<\widetilde{y}_2<f_x^+(\widetilde{y}_1),\ |\widetilde{y}_1|<r(x),
}
where $f^\pm _x$ are smooth on $[-r(x),r(x)]$ and satisfy 
\ekv{a.17}
{
\partial _{\widetilde{y}_1}^kf_x^\pm ={\cal O}_k((\epsilon
r(x))^{1-k}),\ k\ge 1,
}
\ekv{a.18}
{
0< f_x^+-f_x,\, f_x-f_x^-\asymp C\epsilon r(x).
}
Later, we will fix $\epsilon >0$ small enough and write $\gamma 
_r=\widehat{\gamma }_{\epsilon ,C\epsilon r}$ and more generally, $\gamma 
_{\alpha r}=\widehat{\gamma }_{\epsilon ,C\epsilon \alpha r}$.

\par We shall next establish an exponentially weighted estimate for
the Dirichlet Laplacian in $\gamma  _r$:
\begin{prop}\label{a1}
Let $C>0$ be sufficiently large and $\epsilon >0$ sufficiently small in the definition of $\gamma _r$. Then there exists a new constant $C>0$ such that if $\phi \in
C^2(\overline{\gamma  }_r)$ and 
\ekv{cz.7.3}
{
|\phi '_x|\le \frac{1}{Cr},
}
we have 
\ekv{cz.7.6}
{
\Vert e^{\phi }Du\Vert+\frac{1}{C}\Vert \frac{1}{r}e^{\phi }u\Vert
\le C\Vert re^\phi \Delta u\Vert,\ u\in (H_0^1\cap H^2)(\gamma  _r),
}
where $\Vert w\Vert$ denotes the $L^2$ norm when the function $w$  is scalar and we write 
$$
(v|w)=\int\sum v_j(x)\overline{w}_j(x) L(dx),\ \Vert v\Vert=\sqrt{(v|v)},
$$
for ${\bf C}^n$-valued functions with components in $L^2$. $H_0^1$ and $H^2$ are the standard Sobolev spaces.
\end{prop}
\begin{proof} Let $\phi \in C^2(\overline{\gamma 
  }_r;{\bf R})$ and put 
$$
-\Delta _\phi = e^{\phi }\circ (-\Delta )\circ e^{-\phi }=D_x^2-(\phi '_x)^2+i(\phi '_x\circ D_x+D_x\circ \phi '_x),
$$
where we make the usual observation that the last term is formally
anti-self-adjoint. Then for every $u\in (H^2\cap H_0^1)(\gamma  _r)$:
\ekv{cz.7.1}
{
(-\Delta _\phi u|u)=\Vert D_xu\Vert^2-((\phi '_x)^2u|u).
}

\par We need an apriori estimate for $D_x$. Let $v:\overline{\gamma 
}_r\to {\bf R}^n$ be sufficiently smooth. We sometimes consider $v$ as
a vector field. Then for $u\in (H^2\cap H_0^1)(\gamma  _r)$:
$$
(Du|vu)-(vu|Du)=i(\mathrm{div\,}(v)u|u).
$$

Assume $\mathrm{div\,}(v)>0$. Recall that if $v=\nabla w$, then
$\mathrm{div\,}(v)=\Delta w$, so it suffices to take $w$ strictly
subharmonic. Then 
$$
\int \mathrm{div\,}(v)|u|^2dx\le 2\Vert vu\Vert\Vert Du\Vert\le \Vert
Du\Vert^2+\Vert vu\Vert^2,
$$
which we write
$$
\int (\mathrm{div\,}(v)-|v|^2)|u|^2 dx\le \Vert Du\Vert^2.
$$
Using this in (\ref{cz.7.1}), we get
\begin{eqnarray*}
&&\frac{1}{2}\Vert Du\Vert^2+\int
(\frac{1}{2}(\mathrm{div\,}(v)-|v|^2)-(\phi '_x)^2)|u|^2 dx\le\\
&& \Vert \frac{1}{k}(-\Delta _\phi )u\Vert\Vert ku\Vert\le 
\frac{1}{2}\Vert \frac{1}{k}(-\Delta _\phi )u\Vert^2+\frac{1}{2}
\Vert ku\Vert^2,
\end{eqnarray*}
where $k$ is any positive continuous function on $\overline{\gamma 
}_r$. We write this as
\ekv{cz.7.2}
{
\frac{1}{2}\Vert Du\Vert^2+\int
(\frac{1}{2}(\mathrm{div\,}(v)-|v|^2-k^2)-(\phi _x')^2)|u|^2dx \le
\frac{1}{2}\Vert \frac{1}{k}(-\Delta _\phi )u\Vert^2.
}

\par We shall see that we can choose $v$ so that 
\ekv{cz.7.2.3}{
\mathrm{div\,}(v)\ge r^{-2},\ |v|\le {\cal O}(r^{-1}).
}
After replacing $v$ by $C^{-1}v$ for a sufficiently large constant
$C$, we then achieve that 
\ekv{cz.7.2.7}{\mathrm{div\,}(v)-|v|^2\asymp r^{-2}.} 

\par Before continuing, let us establish (\ref{cz.7.2.3}): Let
$g=g_\epsilon $ be the function in the definition of $\gamma 
_r=\widehat{\gamma }_{\epsilon ,C\epsilon r}$ in (\ref{a.15}), so that $C^{-1}\le |\nabla g|\le 1$ (with
the new $C$ independent of $\epsilon $, $C$ in (\ref{a.15})),
$\partial ^\alpha g={\cal O}_\epsilon (r(x)^{1-|\alpha |})$. Put 
\ekv{b.1}
{
v=\nabla (e^{\lambda g/r}),
}
where $\lambda >0$ will be sufficiently large. Notice that 
$$
\nabla (\frac{g}{r})=\frac{\nabla g}{r}-\frac{g\nabla r}{r^2},
$$
where $$|\frac{\nabla g}{r}|\asymp \frac{1}{r}$$ uniformly with respect
to $\epsilon $ and 
$$
|\frac{g\nabla r}{r^2}|={\cal O}(1)\frac{g}{r}\frac{1}{r}={\cal
  O}(\epsilon )\frac{1}{r},
$$
in $\gamma  _r$, so if we fix $\epsilon >0$ sufficiently small,
then 
$$
|\nabla (\frac{g}{r})|\asymp \frac{1}{r}.
$$
We have 
$$
v=e^{\frac{\lambda g}{r}}\lambda \nabla (\frac{g}{r}),\ |v|\asymp
e^{\lambda {\cal O}(\epsilon )}\frac{\lambda }{r},
$$
so the second part of (\ref{cz.7.2.3}) holds for every fixed value of $\lambda $. Further,
$$
\mathrm{div\,}(v)=e^{\frac{\lambda g}{r}}(\lambda ^2|\nabla
(\frac{g}{r})|^2+\lambda \Delta (\frac{g}{r})).
$$
Here,
$$
|\nabla (\frac{g}{r})|^2\asymp \frac{1}{r^2},\ \Delta
(\frac{g}{r})={\cal O}(\frac{1}{r^2}),
$$
so if we fix $\lambda $ large enough, we also get the first part of (\ref{cz.7.2.3}).

If we
choose $k=(Cr)^{-1}$ for a sufficiently large constant $C$, we get
from (\ref{cz.7.2.7}), (\ref{cz.7.3})
$$
\frac{1}{2}(\mathrm{div\,}(v)-|v|^2-k^2)-(\phi '_x)^2\asymp r^{-2}.
$$
Thus, with a new sufficiently large constant $C$, we get from 
(\ref{cz.7.2}):
\ekv{cz.7.4}
{
\Vert Du\Vert^2+\frac{1}{C}\int_{\gamma  _r}\frac{1}{r^2}|u|^2dx\le
C\Vert r(-\Delta _\phi )u\Vert^2,
}
which we can also write as
\ekv{cz.7.5}
{
\Vert Du\Vert+\frac{1}{C}\Vert \frac{1}{r}u\Vert\le C \Vert r(-\Delta
_\phi )u\Vert .
}
Keeping in mind (\ref{cz.7.3}), we get (\ref{cz.7.6}) 
by applying (\ref{cz.7.5}) to $e^{\phi }u$.\end{proof}
 
\par If $\Omega \Subset {\bf C}$ has smooth boundary, let $G_\Omega $,
$P_\Omega $ denote the Green and the Poisson kernels of $\Omega $, so
that the Dirichlet problem,
$$
\Delta u=v,\ {{u}_\vert}_{\partial \Omega }=f,\quad u,v\in C^\infty
(\overline{\Omega }),\ f\in C^{\infty }(\partial \Omega ),
$$
has the unique solution
$$
u(x)=\int_\Omega G_\Omega (x,y)v(y)L(dy)+\int_{\partial \Omega
}P_\Omega (x,y)f(y)|dy|.
$$
Recall that $-G_\Omega \ge 0$, $P_\Omega \ge 0$. It is also clear that 
\ekv{c.1}
{
-G_{\Omega }(x,y)\le C-\frac{1}{2\pi }\ln |x-y|,
}
where $C>0$ only depends on the diameter of $\Omega $. Indeed, let
$-G_0(x,y)$ denote the right hand side of (\ref{c.1}) and choose $C>0$
large enough so that $-G_0\ge 0$ on $\Omega \times \Omega $. Then on
the operator level,
$$
G_{\Omega }v=G_0v-P_\Omega ({{G_0v}_\vert}_{\partial \Omega }),
$$
so that 
$$
G_{\Omega }(x,y)=G_0(x,y)-\int_{\partial \Omega }P_\Omega
(x,z)G_0(z,y) |dz|,
$$
and hence $G_\Omega \ge G_0$, $-G_\Omega \le -G_0$.

\par We will also use the scaling property:
\ekv{c.2}
{
G_{\Omega }(\frac{x}{t},\frac{y}{t})=G_{t\Omega }(x,y),\ x,y\in
t\Omega , t>0,
}
and the fact that $-G_{\Omega }$ is an increasing function of $\Omega $
in the natural sense.
\begin{prop}\label{a2} Under the same assumptions as in Proposition \ref{a1} there exists a (new) constant $C>0$ such that we have 
\ekv{c.3}
{
-G_{\gamma  _r}(x,y)\le C-\frac{1}{2\pi }\ln \frac{|x-y|}{r(y)},\hbox{
  when }|x-y|\le \frac{r(y)}{C},
}
\ekv{c.4}
{
-G_{\gamma  _r}(x,y)\le C \exp (-\frac{1}{C}\int_{\pi _\gamma (y)}^{\pi _\gamma (x)}\frac{1}{r(t)}|dt|),\hbox{
  when }|x-y|\ge \frac{r(y)}{C},
}
where it is understood that the integral is evaluated along $\gamma $
from $\pi _\gamma (y)\in \gamma $ to $\pi _\gamma (x)\in \gamma $,
where $\pi _\gamma (y)$, $\pi _\gamma (x)$ denote points in $\gamma $
with $|x-\pi _\gamma (x)|=\mathrm{dist\,}(x,\gamma )$, 
$|y-\pi _\gamma (y)|=\mathrm{dist\,}(y,\gamma )$, and we choose these
two points (when they are not uniquely defined) and the intermediate
segment in such a way that the integral is as small as possible.
\end{prop}
\begin{proof} 
Let $y\in \gamma  _r$, and put $t=r(y)$. Then we can find $\Omega
\Subset {\bf C}$ uniformly bounded (with respect to $y$) whose boundary is uniformly bounded in the $C^\infty $ sense, such that $\gamma  _r$ coincides with $y+t\Omega =:\Omega _y$
in $D(y,2r(y)/C)$, $\Omega _y\subset D(y,\frac{4r(y)}{C})$ and $r\asymp r(y)$ in that disc. In view of (\ref{c.1}), (\ref{c.2}) we see that 
$-G_{\Omega
  _y}(x,y)$ satisfies the upper bound in (\ref{c.3}). Let $\chi =\chi
(\frac{x-y}{r(y)})$ be a standard cut-off equal to one on $D(y,\frac{r(y)}{C})$ with $\mathrm{supp\,}\chi (\frac{\cdot -y}{r(y)})\subset D(y,\frac{2r(y)}{C})$,
and write the identity:
\ekv{c.5}
{
G_{\gamma  _r}(\cdot ,y)=\chi (\frac{\cdot -y}{r(y)})G_{\Omega
  _y}(\cdot ,y)-G_{\gamma  _r}[\Delta ,\chi (\frac{\cdot
  -y}{r(y)})]G_{\Omega _y}(\cdot ,y).
}
Using that $-G_{\Omega _y}$ satisfies (\ref{c.3}), we see that the
$L^2$-norm of $G_{\Omega _y}(\cdot ,y)$ over the cut-off region
(i.e. the support of the $x$-gradient of the cut-off) is
${\cal O}(r(y))$ and since $G_{\Omega _y}$ is harmonic with boundary
value $0$ there, the $L^2$-norm of $\nabla _xG_{\Omega
  _y}(x,y)$ over the same region is ${\cal O}(1)$. It follows that 
$$
\Vert [\Delta ,\chi (\frac{\cdot -y}{r(y)})]G_{\Omega _y}(\cdot ,y)\Vert
={\cal O}(\frac{1}{r(y)}),
$$
and hence, by applying (\ref{cz.7.6}) with $\phi =0$ to 
$$
u=G_{\gamma  _r}[\Delta ,\chi (\frac{\cdot -y}{r(y)})]G_{\Omega
  _y}(\cdot ,y),
$$ 
we get
$$
\frac{1}{r}G_{\gamma  _r}[\Delta ,\chi (\frac{\cdot
  -y}{r(y)})]G_{\Omega _y}(\cdot ,y)={\cal O}(1),\hbox{ in }L^2(\gamma  _r).
$$
Away from $\mathrm{supp\,}[\Delta ,\chi (\frac{\cdot -y}{r(y)})]$ the
function $G_{\gamma  _r}[\Delta ,\chi (\frac{\cdot -y}{r(y)})]G_{\Omega
_y}(\cdot ,y)$ is harmonic on $\gamma  _r$ with boundary value zero and
we conclude that inside the region where $\chi (\frac{\cdot
  -y}{r(y)})=1$, it is ${\cal O}(1)$. From (\ref{c.5}) we then get the
estimate (\ref{c.3}).

\par To get (\ref{c.4}) we now apply the same reasoning to
(\ref{c.5}), now with $\phi $ as in (\ref{cz.7.3}), (\ref{cz.7.6}),
together with standard arguments for exponentially weighted estimates,
for instance as in \cite{HeSj84}.
\end{proof}

\par We will also need a lower bound on $G_{\gamma  _r}$ on suitable
subsets of $\gamma  _r$. For $\epsilon >0$ fixed and sufficiently small,
we say that $M\Subset \gamma  _r$ is an
elementary piece of $\gamma  _r$ if 
\begin{itemize}
\item $M\subset \gamma  _{(1-\frac{1}{C})r}$,
\item $\frac{1}{C}\le \frac{r(x)}{r(y)}\le C$, $x,y\in M$,
\item $\exists y\in M$ such that $M=y+r(y)\widetilde{M}$, where
  $\widetilde{M}$ belongs to a bounded set of relatively compact
  subsets of ${\bf C}$ with smooth boundary.
\end{itemize}
In the following, it will be tacitly understood that we choose our
elementary pieces with some uniform control ($C$ fixed and uniform
control on the $\widetilde{M}$).
\begin{prop}\label{a3}
If $M$ is an elementary piece in $\gamma  _r$, then 
\ekv{c.6}
{
-G_{\gamma  _r}(x,y)\asymp 1+|\ln \frac{|x-y|}{r(y)}|,\ x,y\in M.
}
\end{prop}
\begin{proof}
We just outline the argument. First, by using arguments from the proof
of Proposition \ref{a2} (without any exponential weights), we see that 
\ekv{c.7}
{
-G_{\gamma  _r}(x,y)\asymp -\ln \frac{|x-y|}{r(y)},\hbox{ when }x,y\in
M,\ \frac{|x-y|}{r(y)}\ll 1.
}
Next, if $M'$ is a slightly larger elementary piece of the form
$y+(1+\frac{1}{C})r(y)\widetilde{M}$, then from Harnack's inequality
for the positive harmonic function $-G_{\gamma  _r}(\cdot ,y)$ on
$M'\setminus D(y,\frac{1}{2C}r(y))$, we see that $-G_{\gamma 
  _r}(x,y)\asymp 1$ in $M\setminus D(y,\frac{1}{C}r(y))$, which
together with (\ref{c.7}) gives (\ref{c.6}).
\end{proof}

\section{Distribution of zeros}\label{di}
\setcounter{equation}{0}

\par Let $\phi $ be a continuous subharmonic function defined in some neighborhood of 
$\overline{\gamma  _r}$. Let 
\ekv{cz.9}
{
\mu =\mu _\phi =\Delta \phi 
}
be the corresponding locally finite  positive measure.

Let $u$ be a holomorphic function defined in a neighborhood of 
$\Gamma \cup\overline{\gamma  _r}$. 
We assume that 
\ekv{cz.10}
{
h\ln \vert u(z)\vert \le \phi (z),\ z\in\overline{\gamma  _r}. 
}

\begin{lemma}\label{cz1}
Let $z_0\in M$, where $M$ is an elementary piece, such that
\ekv{cz.11}
{h\ln \vert u(z_0)\vert \ge \phi (z_0)-\epsilon ,\ 0<\epsilon \ll 1.}
Then the number of zeros of $u$ in $M$ is 
\ekv{cz.12}{\le {C\over h}(\epsilon +\int_{\gamma  _r}-G_{\gamma _r}(z_0,w)\mu (dw)).}
\end{lemma}
\begin{proof}
Writing $\phi $ as a uniform limit of an increasing sequence of smooth 
functions, we may assume that $\phi \in C^\infty $.
Let 
$$
n_u(dz)=\sum 2\pi \delta (z-z_j),
$$
where $z_j$ are the zeros of $u$ counted with their multiplicity. We may 
assume that no $z_j$ are situated on $\partial \gamma  _r$. Then, since 
$\Delta \ln \vert u\vert =n_u$,
\eeeekv{cz.13}
{
h\ln \vert u(z)\vert&=&} {&& \int_{\gamma  _r} G_{\gamma _r}(z,w)h n_u (dw)+\int_{\partial 
\gamma  _r}P_{\gamma _r}(z,w)h\ln \vert u(w)\vert \vert dw\vert}   
{&\le& \int_{\gamma  _r}G_{\gamma _r}(z,w)hn_u(dw)+\int_{\partial \gamma  
_r}P_{\gamma _r}(z,w)\phi (w)\vert dw\vert} 
{&=& \int_{\gamma  _r}G_{\gamma _r}(z,w)hn_u(dw)+\phi (z)-\int_{\gamma  
_r}G_{\gamma _r}(z,w)\mu (dw).}
Putting $z=z_0$ in \no{cz.13} and using \no{cz.11}, we get
$$
\int_{\gamma  _r}-G_{\gamma _r}(z_0,w)hn_u(dw)\le \epsilon +\int_{\gamma  
_r}-G_{\gamma _r}(z_0,w)\mu (dw).
$$
Now $$
-G_{\gamma _r}(z_0,w)\ge {1 \over C},\ w\in M,
$$
and we get \no{cz.12}.
\end{proof}

\par Notice that this argument is basically the same as when using 
Jensen's 
formula to estimate the number of zeros of a holomorphic function in a disc. 

\par Let $z_j^0$, $z_j$ be as in Theorem \ref{int1}. We may arrange so that $\widetilde{\gamma }_{r/C_1}\subset \gamma _r\subset \widetilde{\gamma }_r$. In particular, the assumptions of Theorem \ref{int1} imply (\ref{cz.10}). 
Now we sharpen the assumption \no{cz.11} and assume as in Theorem \ref{int1},
\ekv{cz.14}
{
h\ln \vert u(z_j)\vert \ge \phi (z_j)-\epsilon_j .
}

Let $M_j\subset \gamma _r$ be elementary pieces such that 
\ekv{cz.16.5}
{z_j\in M_j,\ \mathrm{dist\,}(z_j,M_k)\ge \frac{r(z_j)}{C}\hbox{ when
  }
k\ne j,\
  \gamma _{\widetilde{r}}\subset \cup_j M_j,\ \widetilde{r}=(1-\frac{1}{\widetilde{C}})r,}
where $\widetilde{C}\gg 1$. Recall that $\gamma _r=\widehat{\gamma }_{\epsilon ,C\epsilon r}$ where $C,\epsilon $ are now fixed (cf (\ref{a.15}), and that $\gamma _{\alpha r}=\widehat{\gamma }_{\epsilon ,\alpha C\epsilon r}$.
We will also assume for a while that $\phi $ 
is smooth. 

\par According to Lemma \ref{cz1}, we have 
\ekv{cz.17}
{
\# (u^{-1} (0)\cap M_j)\le 
{C_3\over h}(\epsilon_j +\int_{\gamma _r}-G_{\gamma_r}(z_j,w)\mu (dw)).
}

\par Consider the harmonic functions on $\gamma _{\widetilde{r}}$,
\ekv{cz.19}
{
\Psi (z)=h(\ln \vert u(z)\vert +\int_{\gamma 
_{\widetilde{r}}}-G_{\gamma _{\widetilde{r}}}(z,w)n_u(dw)),
}
\ekv{cz.20}
{
\Phi (z)=\phi (z)+\int_{\gamma _{\widetilde{r}}}-
G_{\gamma _{\widetilde{r}}}(z,w)\mu (dw).
}
Then $\Phi (z)\ge \phi (z)$ with equality on $\partial \gamma 
_{\widetilde{r}}$. Similarly, $\Psi (z)\ge h\ln \vert u(z)\vert $ with 
equality on $\partial \gamma _{\widetilde{r}}$.

\par Consider the harmonic function
\ekv{cz.21}
{
H(z)=\Phi (z)-\Psi (z),\ z\in \gamma _{\widetilde{r}}.
}
Then on $\partial \gamma _{\widetilde{r}}$, we have by \no{cz.10} that
$$
H(z)=\phi (z)-h\ln \vert u(z)\vert \ge 0,
$$
so by the maximum principle,
\ekv{cz.22}
{
H(z)\ge 0,\hbox{ on }\gamma _{\widetilde{r}}.
}
By \no{cz.14}, we have 
\eeeekv{cz.23}
{
H(z_j)&=& \Phi (z_j)-\Psi (z_j)
}
{
&=&\phi (z_j)-h\ln \vert u(z_j)\vert }
{&&+\int_{\gamma _{\widetilde{r}}}-G_{\gamma _{\widetilde{r}}}(z_j,w)\mu 
(dw)-\int_{\gamma _{\widetilde{r}}} -G_{\gamma _{\widetilde{r}}}(z_j,w)hn_u(dw)
}
{
&\le & \epsilon_j +\int_{\gamma _{\widetilde{r}}} -
G_{\gamma _{\widetilde{r}}}(z_j,w)\mu (dw).
}

\par Harnack's inequality implies that 
\ekv{cz.24}
{
H(z)\le {\cal O}(1)(\epsilon_j +\int -G_{\gamma _{\widetilde{r}}}(z_j,w)\mu (dw))\hbox{ 
on }M_j\cap \gamma _{\widehat{r}},\ \widehat{r}=(1-\frac{1}{\widetilde{C}})\widetilde{r}.}

\par Now assume that $u$ extends to a holomorphic function in a neighborhood of 
$\Gamma \cup \overline{\gamma _r}$. We then would like to evaluate the 
number of zeros of $u$ in $\Gamma $. Using \no{cz.17}, we first have 
\ekv{cz.25}
{
\# (u^{-1}(0)\cap \gamma _{\widetilde{r}})\le {C\over h}\sum_{j=1}^N
\left(\epsilon _j
+ \int_{\gamma _r}-G_{\gamma_r}(z_j,w)\mu (dw)\right).
}
\par Let $\chi \in C_0^\infty (\Gamma \cup \gamma _{\widehat{r}};[0,1])$ be equal to 1 on $\Gamma $. Of course $\chi $ 
will have to depend on $r$ but we may assume that for all $k \in{\bf N}$,
\ekv{cz.26}
{
\nabla ^k\chi ={\cal O}(r^{-k}).
}
We are interested in 
\ekv{cz.27}
{
\int \chi (z)hn_u(dz)=\int_{\gamma _{\widehat{r}}}h\ln \vert u(z)\vert 
\Delta \chi (z)L(dz).
}
Here we have on $\gamma _{\widetilde{r}}$
\eeeekv{cz.28}
{h\ln \vert u(z)\vert &=&\Psi (z)-\int_{\gamma 
_{\widetilde{r}}}-G_{\gamma _{\widetilde{r}}}(z,w)hn_u(dw)}
{&=&\Phi (z)-H(z)-\int_{\gamma 
_{\widetilde{r}}}-G_{\gamma _{\widetilde{r}}}(z,w)hn_u(dw)}
{
&=&\phi (z)+\int_{\gamma _{\widetilde{r}}}-G_{\gamma _{\widetilde{r}}}
(z,w)\mu 
(dw)-H(z)-\int_{\gamma _{\widetilde{r}}}-G_{\gamma _{\widetilde{r}}}(z,w)hn_u(dw)
}
{&=& \phi (z)+R(z),}
where the last equality defines $R(z)$.

\par Inserting this in \no{cz.27}, we get 
\ekv{cz.29}
{
\int\chi (z)hn_u(dz)=\int \chi (z)\mu (dz)+\int R(z)\Delta \chi (z)L(dz).
}
(Here we also used some extension of $\phi $ to $\Gamma $ 
with $\mu =\Delta \phi $.) The task is now to estimate $R(z)$ and the 
corresponding integral in \no{cz.29}. Put 
\ekv{cz.30}
{
\mu _j=\mu (M_j\cap \gamma _{\widetilde{r}}).
}
Using the exponential decay property \no{c.4} (equally valid for 
$G_{\gamma _{\widetilde{r}}}$) we get for $z\in M_j\cap \gamma 
_{\widetilde{r}}$, ${\rm dist\,}(z,\partial M_j)\ge r(z_j)/{\cal
  O}(1)$:
\ekv{cz.31}
{
\int_{\gamma _{\widetilde{r}}}-G_{\gamma _{\widetilde{r}}}
(z,w)\mu (dw)\le 
\int_{M_j\cap \gamma _{\widetilde{r}}}
-G_{\gamma _{\widetilde{r}}}(z,w)\mu 
(dw)+{\cal O}(1)\sum_{k\ne j}\mu _ke^{-{1\over C_0}\vert j-k\vert },
}
where $|j-k|$ denotes the natural distance from $j$ to $k$ in ${\bf
  Z}/N{\bf Z}$ and $C_0>0$.
Similarly from \no{cz.24}, we get 
\ekv{cz.32}
{
H(z)\le {\cal O}(1)(\epsilon_j +
\int_{M_j\cap\gamma_{\widetilde{r}}}-G_{\gamma _{\widetilde{r}}}
(z_j,w)\mu (dw)+\sum_{k\ne j}e^{-{1\over 
C_0}\vert j-k\vert }\mu _k),
}
for $z\in M_j\cap \gamma _{\widetilde{r}}$. 

\par This gives the following estimate on the contribution from the first 
two terms in $R(z)$ to the last integral in \no{cz.29}:
\eeekv{cz.32.5}
{
&&\int_{\gamma _{\widetilde{r}}}\left(\int_{\gamma _{\widetilde{r}}}
-G_{\gamma _{\widetilde{r}}}(z,w)\mu (dw)-H(z)\right)\Delta \chi (z)L(dz)
}{&&={\cal O}(1)\sum_j(\epsilon _j+\int_{M_j\cap\gamma_{\widetilde{r}}}-G_{\gamma
    _{\widetilde{r}}}(z_j,w)\mu (dw))+\sum_{k\ne
    j}e^{-\frac{1}{C_0}|j-k|}\mu _k)
}{&&+{\cal O}(1)\sum_j\int_{M_j\cap\gamma_{\widetilde{r}}}\int_{M_j\cap\gamma_{\widetilde{r}}}-G_{\gamma _{\widetilde{r}}}
(z,w)\mu (dw)|\Delta \chi (z)|L(dz).
}
Here,
\ekv{cz.32.7}{
\int_{M_j\cap\gamma_{\widetilde{r}}}-G_{\gamma _{\widetilde{r}}}(z,w) |\Delta \chi (z)|L(dz)={\cal O}(1), 
}
so (\ref{cz.32.5}) leads to 
\eekv{cz.33}
{
&&\int_{\gamma _{\widetilde{r}}}\left(\int_{\gamma _{\widetilde{r}}}
-G_{\gamma _{\widetilde{r}}}(z,w)\mu (dw)-H(z)\right)
\Delta \chi (z)L(dz)}
{&&={\cal O}(1)\left(\mu (\gamma _{\widetilde{r}})+\sum_j \epsilon _j
+\sum_j \int_{M_j\cap\gamma_{\widetilde{r}}} -G_{\gamma _{\widetilde{r}}}(z_j,w)\mu (dw)\right).
}

\par The contribution from the last term in $R(z)$ (in \no{cz.28}) to the 
last integral in \no{cz.29} is 
\ekv{cz.34}
{
\int_{z\in \gamma _{\widehat{r}}}\int_{w\in \gamma 
_{\widetilde{r}}}G_{\gamma _{\widetilde{r}}}(z,w)hn_u(dw)\Delta \chi (z)L(dz).
}
Here, by using an estimate similar to (\ref{cz.31}), with $\mu (dw)$
replaced by $L(dz)$, together with (\ref{cz.32.7}), we get 
$$\int_{z\in \gamma _{\widehat{r}}}G_{\gamma _{\widetilde{r}}}(z,w)(\Delta \chi )
(z)L(dz)={\cal O}(1),
$$
so the expression \no{cz.34} is by (\ref{cz.25})
\eeekv{cz.35}
{
&&{\cal O}(h)\# (u^{-1} (0)\cap \gamma _{\widetilde{r}})
}
{
&=&{\cal O}(1)\sum_{j=1}^N (\epsilon _j+ \int_{\gamma 
_r}(-G_{\gamma_r}(z_j,w))\mu (dw))
}
{&=&
{\cal O}(1)(\mu (\gamma _r)+\sum_{j=1}^N(\epsilon _j+ \int_{M 
_j}-G_{\gamma_r}(z_j,w)\mu (dw))).
}
This is quite similar to (\ref{cz.33}). Using Proposition \ref{a2}, we have 
\begin{eqnarray*}
&&\int_{M_j\cap\gamma_{\widetilde{r}}}-G_{\gamma _{\widetilde{r}}}(z_j,w)\mu (dw)\le\\
&&{\cal O}(1)(\int_{|w-z_j|\le \frac{r(z_j)}{C}}|\ln
\frac{|z_j-w|}{r(z_j)}|\mu (dw)+\mu (M_j\cap\gamma_{\widetilde{r}}))
\end{eqnarray*}
and similarly for the last integral in (\ref{cz.35}).
Using all this in \no{cz.29}, we get
\eekv{cz.36}
{&&\hskip -5truemm
\int\chi (z)hn_u(dz)=\int \chi (z)\mu (dz)
}
{
&&+{\cal O}(1)(\mu (\gamma _r)+\sum_j(\epsilon _j+\int _{|w-z_j|\le
  r(z_j)/C}|\ln (\frac{|z_j-w|}{r(z_j)})|\mu (dw)
).
}
We replace the smoothness assumption on $\phi $ by the assumption that 
$\phi $ is continuous near $\Gamma $ and keep \no{cz.14}. Then by 
regularization, we still get \no{cz.36}.

Now we observe that 
$$
|\# (u^{-1}(0)\cap \Gamma ) -\frac{1}{2\pi h}\int \chi (z) hn_u (dz)|
\le
\# (u^{-1}(0)\cap \gamma _{\widetilde{r}}),
$$
which can be estimated by means of (\ref{cz.35}),
and combining this with (\ref{cz.36}), we get 
\eekv{d.2}
{
&&|\# (u^{-1}(0)\cap \Gamma )-\frac{1}{2\pi h}\mu (\Gamma )|\le
}
{&& \frac{{\cal O}(1)}{h}
\left( \mu (\gamma _r)+\sum_j (\epsilon _j +\int_{|w-z_j|\le
    \frac{r(z_j)}{C}} |\ln \frac{|z_j-w|}{r(z_j)}| \mu (dw))
\right) .}
This completes the proof of Theorem \ref{int1}.\hfill{$\Box$}

\medskip\par
We next discuss when the contribution from the
logarithmic integrals in (\ref{int.8}) can be eliminated or simplified. 
Let $r$, $C_1$, $z_j^0$ be as in Theorem \ref{int1}.
Using the estimates above, we get 
\begin{eqnarray*}
&&\int_{D(z_j^0,\frac{r(z_j^0)}{2C_1})}\int
_{D(z,\frac{r(z)}{4C_1})}|\ln \frac{|w-z|}{r(z)}|\mu (dw)\frac{L(dz)}{L(D(z_j^0,\frac{r(z_j^0)}{2C_1}))}\le\\
&&\int_{D(z_j^0,\frac{r(z_j^0)}{2C_1})}\int
_{D(z_j^0,\frac{r(z_j^0)}{C_1})}|\ln \frac{|w-z|}{r(z)}|\mu
(dw)\frac{L(dz)}{L(D(z_j^0,\frac{r(z_j^0)}{2C_1}))}\le\\
&&{\cal O}(1) \mu (D(z_j^0,\frac{r(z_j^0)}{C_1})),
\end{eqnarray*}
where we changed the order of integrations in the last step and also used that $r(z)\asymp r(z_j^0)$ in
$D(z_j^0,\frac{r(z_j^0)}{C_1})$. We conclude that the
mean-value of 
$$
D(z_j^0,\frac{r(z_j^0)}{2C_1})\ni z\mapsto \int_{D(z,\frac{r(z)}{4C_1})}
|\ln \frac{|w-z|}{r(z)}|\mu (dw)
$$
is ${\cal O}(1)\mu (D(z_j^0,\frac{r(z_j^0)}{C_1}))$. Thus we can find
$\widetilde{z}_j\in D(z_j^0,\frac{r(z_j^0)}{2C_1})$ such that 
$$
\sum_{j=1}^N \int_{D(\widetilde{z}_j,\frac{r(\widetilde{z}_j)}{4C_1})}
|\ln \frac{|w-\widetilde{z}_j|}{r(\widetilde{z}_j)}|\mu (dw)={\cal
  O}(1)\mu (\widetilde{\gamma }_{r}).
$$
This gives Theorem \ref{int2}.\hfill{$\Box$}

\medskip\par
For completeness, we recall (a slight extension of) the counting proposition in \cite{HaSj}. Assume for simplicity that $\Gamma $ is an $h$ independent Lipschitz domain as defined in the introduction and that $\phi $ is subharmonic function, defined in a fixed neighborhood of the boundary
Assume that $\mu =\Delta \phi $. Let $\rho _0\in ]0,2]$ and assume that for all discs $D(z,t)$ contained in a fixed neighborhood of $\gamma =\partial \Gamma $,
\ekv{cz.37}
{
W_z(t):=\mu (D(z,t))={\cal O}(t^{\rho _0}),
}
uniformly with respect to $z,t$.
\begin{remark}\label{cz1.5}\rm
It is easy to see (\cite{HaSj}) that this assumption on $\Delta \phi $ implies that $\phi $ 
is continuous near $\Gamma $.
\end{remark}
As in \cite{HaSj}, we have
\begin{lemma}\label{cz2}
Assume \no{cz.37} for some $\rho _0\in ]0,2]$. Then for $0<2t<\tilde{r}\ll 1$
and $z\in {\bf C}$ for which $D(z,\tilde{r})$ belongs to the fixed
neighborhood of $\gamma $, where $\mu $ is defined and fulfills (\ref{cz.37}),
we have 
\ekv{cz.39}
{
\int_{D(z,\tilde{r})}|\ln \frac{|z-w|}{\tilde{r}}|\mu (dw)\le {\cal O}(1)t^{\rho _0}\ln {\tilde{r}\over 
t}+{\cal O}(1)\ln({\tilde{r}\over t})\mu (D(z,\tilde{r})).
}
\end{lemma}
\begin{proof}
This follows from the estimates,
$$
\int_{D(z,\tilde{r})\setminus D(z,t)}|\ln \frac{|z-w|}{\tilde{r}}|\mu (dw)\le {\cal
  O}(1) (\ln \frac{\tilde{r}}{t})\mu (D(z,\tilde{r})),
$$
and
\begin{eqnarray*}\int_{D(z,t)}|\ln \frac{|z-w|}{\tilde{r}}|\mu (dw)&\le &{\cal O}(1)\int_0^t \ln {\tilde{r}\over 
s}dW_z(s)\\
&=&{\cal O}(1)([\ln ({\tilde{r}\over s})W_z(s)]_0^t+\int_0^t {1\over s}W_z(s)ds)\\
&=&{\cal O}(1) t^{\rho _0}\ln {\tilde{r}\over t}.
\end{eqnarray*}
\end{proof}

\begin{cor}\label{cz3}
Under the same assumptions, we have for every $N\in{\bf N}$:
\ekv{cz.40}
{
\int_{D(z,\tilde{r})} |\ln \frac{|z-w|}{\tilde{r}}|\mu (dw)\le {\cal O}_N(1)(\tilde{r}^N+
\ln ({1\over \tilde{r}})\mu 
(D(z,\tilde{r}))).
}
\end{cor}
\begin{proof}
We just choose $t=\tilde{r}^M$, $0<M\in{\bf N}$ and use that $\ln \tilde{r}^{-M}=M\ln \tilde{r}^{-1}$.
\end{proof}

We now get:
\begin{theo}\label{d4}
Assume that (\ref{cz.37}) holds for all discs $D(z,t)$ contained in
some fixed neighborhood of $\gamma $. Then under the assumtions of
Theorem \ref{int1}, we have for every $N\in {\bf N}$:
\eekv{d.11}
{
&&|\# (u^{-1}(0)\cap \Gamma )-\frac{1}{2\pi h}\mu (\Gamma )|\le
}
{&&
\frac{\widetilde{C}}{h}\left( 
\sum_1^N (\epsilon _j+{\cal O}_N(r(z_j)^N)+{\cal O}_N(1)\int
_{\widetilde{\gamma}_{r}}\ln \frac{1}{r(z)}\mu (dz))
\right).
}
\end{theo}

\section{Application to sums of exponential functions}\label{ap}
\setcounter{equation}{0} Consider the function \ekv{ap.1}{ u(z;h)=\sum_1^N e^{\phi _j(z)/h}, } where $N$ is finite and $\phi _j$ are holomorphic in the open set $\Omega \subset {\bf C}$ and independent of $h$ for simplicity. Put \ekv{ap.2}{ \psi _j(z)=\Re \phi _j(z), } let $\Gamma \Subset \Omega $ have $C^\infty $ boundary $\gamma $ and assume \eekv{ap.3}{ &&\forall x\in \gamma ,\ \Psi (x):=\max_j \psi _j(x)\hbox{ is attained}}{&&\hbox{for at most 2 different values of }j, } \eekv{ap.4} { &&\hbox{If }x\in \gamma ,\ \Psi (x)=\psi _j(x)=\psi _k(x),\ j\ne k,}{&&\hbox{then }\nu (x,\partial_x)(\psi _j(x)-\psi _k(x))\ne 0, } where $\nu $ denotes the normalized vector field (say positively oriented) that is tangent to $\gamma $. We shall see that Theorem \ref{int2} allows us to determine the number of zeros of $u$ in $\Gamma $ up to ${\cal O}(1)$. This result can be further strengthened by using direct arguments (see for instance \cite{HiSj3a}), but the purpose of this section is simply to illustrate the results above. We also notice that the results will be valid if $u$ is holomorphic in $\Omega $ but with the representation (\ref{ap.1}) and the $\phi _j$ defined only in a neighborhood of $\gamma $.

We shall establish the following result (without any claim of novelty, see
\cite{HiSj3a} as well as \cite{Da03, BlMa06}). For a closely related old result on entire functions, see \cite{Le80}, Chapter VI, Section 3, Theorem 9, attributed to A.~Pfluger \cite{Pf38}.
\begin{prop}\label{ap1}
We have
\ekv{ap.5}
{
|\# (u^{-1}(0)\cap \Gamma )-\frac{1}{2\pi h}\int_{\Gamma }\Delta \Psi
(z)L(dz)|= {\cal O}(1).
}
\end{prop}

Here, in the case when $\psi _j$ and $\Psi $ are defined only in a
neighborhood of $\gamma $, we take any distribution extension of $\Psi
$ to a neighborhood of $\Gamma $. Notice that near $\gamma $ the
function $\Psi $ is subharmonic 
and $\Delta \Psi $ is supported by the union
of the curves $\gamma _{j,k}$. On each such curve, $\Delta
\Psi =\frac{\partial }{\partial n}(\psi _j-\psi _k)|dz|$, where 
$n$ is the unit normal to $\gamma _{j,k}$, oriented so that 
$\frac{\partial }{\partial n}(\psi _j-\psi _k)>0$.

\par
We shall prove Proposition \ref{ap1} by means of Theorem
\ref{int2}.

\par Put
\ekv{ap.6}
{
\Phi (z)=h\ln (\sum_1^N e^{\psi _j(z)/h}),\ z\in
\mathrm{neigh\,}(\gamma ),
}
so that $\Phi (z)=hf(\frac{\psi _1}{h},..,\frac{\psi _N}{h})$, where 
\ekv{ap.7}
{
f(x)=\ln (\sum_1^N e^{x_j}).
}
If we define $\theta _j=e^{x_j}/\sum e^{x_k}$, then $\theta
_j>0$, $\theta _1+..+\theta _N=1$, and
\ekv{ap.8}
{
\partial _{x_j} f(x)=\theta _j,
}
\ekv{ap.9}
{
f''(x)=\mathrm{diag\,}(\theta _j)-(\theta _j\theta _k)_{j,k}.
}
For $y\in {\bf R}^N$, we have 
$$
\langle f''(x)y,y\rangle =\sum \theta _jy_j^2-(\sum \theta _jy_j)^2,
$$
which is $\ge 0$, since the function $t\mapsto t^2$ is convex. Hence
$f$ is convex.

\par We apply this to $\Phi (z)$, now with $\theta _j=e^{\psi
  _j(z)/h}/\sum_k e^{\psi _k/h}$, and get
\ekv{ap.10}
{
\partial _z\Phi (z)=\sum \theta _j\partial _{z}\psi _j,\quad \partial
_z=\frac{1}{2}(\partial _{\Re z}+\frac{1}{i}\partial _{\Im z})
}
\ekv{ap.11}
{
\partial _{\overline{z}}\partial _z\Phi (z)=\frac{1}{h}\langle
f''\partial _z\psi ,\partial _{\overline{z}}\psi \rangle= \frac{1}{h}
(\sum\theta _j|\partial _z\psi _j|^2-|\sum \theta _j\partial _z\psi _j|^2).
}
In the last calculation, we also used that $\psi _j$ are
harmonic. It follows that $\Phi $ is subharmonic near $\gamma
$. Also notice that 
\ekv{ap.12}
{
\Delta \Phi (z)={\cal O}(h^{-1}),
}
and that this estimate can be considerably improved away from the
union of the $\gamma _{j,k}$: Assume for instance that $\Psi (z)=\psi
_1\ge \max_{j\ne 1}\psi _j+\delta $, where $\delta >0$ and notice that
we can take $\delta =C^{-1}d(x)$ with $d(x):=\mathrm{dist\,}(z,\cup
\gamma _{j,k})$, by (\ref{ap.3}), (\ref{ap.4}). Then 
\ekv{ap.13}
{
\Phi =h\ln (e^{\psi _1/h}(1+{\cal O}(e^{-\delta /h})))=\psi _1+{\cal
  O}(he^{-\delta /h}).
}
Further, 
\ekv{ap.14}
{
\theta _1=1+{\cal O}(e^{-\delta /h}),\ \theta _j={\cal O}(e^{-\delta
  /h}) \hbox{ for }j\ne 1,
}
so 
\ekv{ap.15}
{
\partial _z\Phi =\partial _z\psi _1+{\cal O}(e^{-\delta /h}),
}
$$
f''(e^{\psi _1/h},..,e^{\psi _N/h})={\cal O}(e^{-\delta /h}),
$$
\ekv{ap.16}
{
\partial _{\overline{z}}\partial _z\Phi ={\cal
  O}(\frac{1}{h}e^{-\delta /h}).
}

\par We will always be able to express the final result in terms of
the simpler function $\Psi $:
\begin{lemma}\label{ap2}
We have 
\ekv{ap.17}
{
\int_\Gamma \Delta \Phi L(dz)-\int_\Gamma \Delta \Psi L(dz)={\cal O}(h).
}
\end{lemma}
\begin{proof}
Using Green's formula, the left hand side of (\ref{ap.17}) can be
written 
$$
\int_{\gamma }(\frac{\partial \Phi }{\partial n}-
\frac{\partial \Psi }{\partial n}) |dz|,
$$
where $n$ is the suitably oriented normal direction. It then suffices
to apply (\ref{ap.15}), with $\psi _1$ replaced by $\Psi $, in the region where $d(z)\gg h$ and use that
the gradients of $\Phi $, $\Psi $ are ${\cal O}(1)$.
\end{proof}

We next notice that 
\ekv{ap.18}
{
h\ln |u(z;h)|\le \Phi (z)
}
in neighborhood of $\gamma $. On the other hand, for $z$ near $\gamma $,
$d(z)\gg h$, we have 
\ekv{ap.19}
{
h\ln |u(z;h)|\ge \Phi (z)-{\cal O}(h)e^{-d(z)/(Ch)}.
}
We can now apply Theorem \ref{int2} with $r=\mathrm{Const.\,}h$,
$d(z_j^0)\ge Ch$, $\epsilon _j={\cal O}(he^{-d(z_j^0)/(Ch)})$, $\phi =\Phi
$. In view of (\ref{ap.12}), (\ref{ap.16}), we see that $\mu
(\widetilde{\gamma }_r)={\cal O}(h)$, $\sum \epsilon _j={\cal O}(h)$,
so 
$$
\# (u^{-1}(0)\cap \Gamma )-\frac{1}{2\pi h}\int_{\Gamma }\Delta \Phi
L(dz)={\cal O}(1),
$$
and we obtain Proposition \ref{ap1} from Lemma
\ref{ap2}. 
\hfill{$\Box$}

\medskip If we would like to work directly with $\phi =\Psi $, we
still have (\ref{ap.19}) with $\Phi $ replaced by $\Psi $, while the
upper bound (\ref{ap.18}) has to be replaced by 
$$
h\ln |u(z)|\le \Psi (z)+Ch,
$$
so we have to take $\phi =\Psi +Ch$ and at most places $\epsilon _j\asymp h$. The effect of that
deterioration can be limited by chosing the $z_j$ more sparcely away
from the union of the $\gamma _{j,k}$, but we can hardly avoid a
remainder ${\cal O}(\ln \frac{1}{h})$ in (\ref{ap.5}). Similarly, by
working with $\phi =\Phi $ and the weaker Theorem \ref{d4}
(essentially from \cite{HaSj}) we also seem to get additional
logarithmic losses.

\section{A simple application to entire functions}\label{ent}
\setcounter{equation}{0}

In this section we give as a simple application a result on the number of zeros in truncated sectors for entire functions with a certain exponential growth. We have been inspired by Theorem 3 in Section 3 of Chapter 3 in \cite{Le80}, there attributed to Pfluger \cite{Pf38}. We give a variant which is not quite identical. 

\par Let $\theta <\vartheta<\theta +2\pi  $ and let $u$ be an entire function or more generally a holomorphic function defined in the sector
\ekv{ent.1}
{
{\bf R}_+e^{i]\theta -\epsilon_0 ,\vartheta+\epsilon _0[},
}
for some fixed $\epsilon _0>0$. Let $\phi $ be a continuous subharmonic function in the same sector, positively homogeneous of degree $\rho >0$, so that $\phi =r^\rho g(\omega )$ in polar coordinates $z=re^{i\omega }$.
Then 
\ekv{ent.2}
{
\Delta \phi =r^{\rho -2}(\rho ^2g(\omega )+g''(\omega )),
}
so the subharmonicity means that 
\ekv{ent.3}
{\nu :=
g''(\omega )+\rho ^2g(\omega )\ge 0 \hbox{ on }]\theta -\epsilon _0,\vartheta +\epsilon _0 [,
}
in the sense of distributions. 

\par Assume that in the sector (\ref{ent.1}), we have
\ekv{ent.4}
{
\ln |u(z)|\le \phi (z)+o(|z|^\rho ),\ |z|\to \infty .
}
Also assume that for all $0<\widetilde{\epsilon} ,\epsilon \ll 1$ there exists $R(\widetilde{\epsilon} ,\epsilon )>0$, such that for every $z$ in the sector (\ref{ent.1}) with $|z|\ge R(\epsilon ,\widetilde{\epsilon} )$, there exists $\widetilde{z}=\widetilde{z}(\widetilde{\epsilon} ,\epsilon )$ in the same sector such that 
\ekv{ent.5}
{
|\widetilde{z}-z|\le \epsilon |z|,\ \ln |u(\widetilde{z})|\ge \phi (\widetilde{z})-\widetilde{\epsilon }|\widetilde{z}|^\rho .
}
In the proof below we shall take $\widetilde{\epsilon }=\epsilon ^2$ where the exponent 2 could be replaced by any exponent $>1$. Write $R(\epsilon )$ instead of $R(\epsilon ,\epsilon ^2)$.

\begin{prop}\label{ent1}
We make the assumptions above. Also assume that the positive measure $\nu $ in (\ref{ent.3}) does not charge $\theta ,\vartheta $ in the sense that 
$$
\nu ([\theta -\epsilon ,\theta +\epsilon ]), \nu ([\vartheta -\epsilon ,\vartheta +\epsilon ]) \to 0,\ \epsilon \searrow 0.
$$
Then the number $n(R;\theta, \vartheta)$ of zeros of $u$ in $]1,R[e^{i[\theta ,\vartheta ]}=:\Gamma (R)$ satisfies
\ekv{ent.6}
{
n(R;\theta ,\vartheta )=\frac{1}{2\pi }(\int_{\Gamma (R)}\Delta (\phi )L(dz)+o(R^\rho )),\ R\to \infty . 
}\end{prop}

\begin{proof}
We apply Theorem \ref{int1} with $h=1$, $r(z)=\epsilon |z|$, and $\Gamma =\widetilde{\Gamma }(R)$ equal to the truncated sector $[R(\frac{\epsilon }{2C_1}),R]e^{i[\theta ,\vartheta ]}$ for $R>R(\frac{\epsilon }{2C_1})$. (Here we notice that our domains and $r$ satisfy the assumptions of the theorem uniformly and $C_1$ is the corresponding uniform constant there.) We choose $z_j^0$ as in Theorem \ref{int1} and let $z_j=\widetilde{z}(s_j^0)$ be as in (\ref{ent.5}) with $\epsilon $ replaced by $\epsilon /(2C_1)$, so that
\ekv{ent.7}
{
|z_j-\widetilde{z}_j|\le \frac{r(z_j)}{2C_1},\ \ln |u(z_j)|\ge \phi (z_j)-\frac{\epsilon ^2}{4C_1^2}|z_j|^\rho .
}
Thus we can take $\epsilon _j$ in Theorem \ref{int1} equal to $\frac{\epsilon ^2}{4C_1^2}|z_j|^\rho$.

We have $\widetilde{\gamma }_r=\cup_{z\in \partial \widetilde{\Gamma } (R)}D(z,\epsilon |z|)$ and since $\mu =\Delta \phi (z)L(dz)$, some straight forward estimates show that
\eeekv {ent.8}
{\mu (\widetilde{\gamma }_r)&\le&C_\rho (\nu ([\theta -2\epsilon ,\theta +2\epsilon ])+}{&&\nu ([\vartheta -2\epsilon ,\vartheta + 2\epsilon ])+\epsilon \nu (\vartheta -2\epsilon ,\theta +2\epsilon ]))R^\rho}{&=& o(1)R^\rho ,\quad \epsilon \to 0.}
Here we also used the assumption that $\nu $ does not charge $\theta $ and $\vartheta $.

\par Next we estimate $\sum \epsilon_j $. The number of points $z_j^0$ on the arc $Re^{i[\theta ,\vartheta ]}$ is ${\cal O}(1/\epsilon )$, so the sum of the corresponding $\epsilon _j$ is ${\cal O}(\epsilon )R^\rho $. The points $z_j^0$ on $[\widetilde{R}(\epsilon ),R]e^{i\theta }$ can be chosen as a geometric progression so that $|z_j^0|=\widetilde{R}(\epsilon )\exp (\epsilon j/C)$ for a suitable constant $C>0$. Here $j\ge 0$ is bounded from above by the requirement $\epsilon j/C\le \ln (R/\widetilde{R}(\epsilon ))$, so the corresponding sum of $\epsilon _j$ is bounded by ${\cal O}(1)\epsilon R^\rho $. The two other parts of the boundary can be treated similarly and we conclude that
\ekv{ent.9}
{
\sum \epsilon _j={\cal O}(1)\epsilon R^\rho .
}

We finally have to treat the logarithmic integrals in (\ref{int.8}). With $r_j=|z_j|$, $\theta _j=\mathrm{arg\,}z_j$, we have $r(z_j)=\epsilon r_j$ and the integral can be bounded by 
\begin{eqnarray*}
\int_{|\omega -\theta _j|\le \frac{\epsilon }{4C_1}}\int_{|r-r_j|\le \frac{\epsilon r_j}{4C_1}}|\ln \frac{|re^{i\omega }-r_je^{i\theta _j}|}{\epsilon r_j}|r^{\rho -1}dr \nu (d\omega ) \le \\
\int_{|\omega -\theta _j|\le \frac{\epsilon }{4C_1}}\int_{|r-r_j|\le \frac{\epsilon r_j}{4C_1}}|\ln \frac{|r-r_j|}{\epsilon r_j}|r^{\rho -1}dr \nu (d\omega )
\end{eqnarray*}
Here the presence of the logarithmic factor does not change the order of magnitude of the last integral with respect to $r$ so we conclude that the integral in (\ref{int.8}) is bounded by 
$$
{\cal O}(1)\mu (\{ 
z\in {\bf C}; |z-|z_j||\le \frac{\epsilon |z_j|}{4C_1},\ |\mathrm{arg\,}z-\mathrm{arg\,}z_j|\le \frac{\epsilon }{4C_1}
\})
$$
and the domains appearing here have at most a fixed finite number of overlaps at each point. Consequently the sum of the logarithmic integrals in (\ref{int.8}) can be bounded from above by ${\cal O}(1)\mu (\widetilde{\gamma }_r)$.

It now follows from Theorem \ref{int1} that the number of zeros of $u$ in 
$[R(\frac{\epsilon }{2C_1}),R]e^{i[\theta ,\vartheta ]}$ is equal to 
$$
\frac{1}{2\pi }(\int _{[R(\frac{\epsilon }{2C_1}),R]e^{i[\theta ,\vartheta ]}}\Delta \phi (z)L(dz)+o(R^\rho ))
$$
when $\epsilon \to 0$, uniformly when $R\ge R(\frac{\epsilon }{2C_1})$. The proposition follows.
\end{proof}

Using (\ref{ent.2}), (\ref{ent.3}), we can write the right hand side of (\ref{ent.6}) as 
$$
\frac{1}{2\pi }R^\rho (o(1)+\rho \int_\theta ^\vartheta g(\omega )d\omega +\frac{1}{\rho }(g'(\vartheta )-g'(\theta )). 
$$

In \cite{Le80} the case of certain variable $\rho $, so called proximate orders, is also treated. We believe that the discussion above can be extended to cover that case.

\end{document}